\documentclass[12pt,leqno,twoside]{amsart}

\usepackage[latin1]{inputenc}
\usepackage[T1]{fontenc}
\usepackage[colorlinks=true, pdfstartview=FitV, linkcolor=blue, citecolor=blue, urlcolor=blue]{hyperref}
\usepackage{amstext,amsmath,amscd, bezier,indentfirst,amsthm,amsgen,enumerate, geometry}
\usepackage[all,knot,arc,import,poly]{xy}
\usepackage{amsfonts,color}
\usepackage{amssymb}
\usepackage{latexsym}
\usepackage{epsfig}
\usepackage{graphicx}
\usepackage{srcltx}
\usepackage{enumitem}

\newtheorem{theorem}{Theorem}[section]
\newtheorem{corollary}[theorem]{Corollary}
\newtheorem{lemma}[theorem]{Lemma}
\newtheorem{proposition}[theorem]{Proposition}
\theoremstyle{definition}
\newtheorem{definition}[theorem]{Definition}
\theoremstyle{remark}
\newtheorem{remark}[theorem]{\sc Remark}
\newtheorem{example}[theorem]{\sc Example}

\renewcommand{\Box}{\square}    

\newcommand{\Sing}{{\rm{Sing\hspace{2pt}}}}

\renewcommand{\Re}{{\mathrm{Re}}}

\newcommand{\im}{{\rm{Im\hspace{2pt}}}}

\newcommand{\ity}{{\infty}}

\newcommand{\m}{\setminus}

\newcommand{\fin}{\hspace*{\fill}$\Box$\vspace*{2mm}}

\newcommand{\cC}{{\mathcal C}}

\newcommand{\cF}{{\mathcal F}}

\newcommand{\cO}{{\mathcal O}}

\newcommand{\cS}{{\mathcal S}}

\newcommand{\cN}{{\mathcal N}}

\newcommand{\bB}{{\mathbb B}}
\newcommand{\bR}{{\mathbb R}}
\newcommand{\bC}{{\mathbb C}}
\newcommand{\bP}{{\mathbb P}}
\newcommand{\bN}{{\mathbb N}}

\newcommand{\bQ}{{\mathbb Q}}

\begin{document}

\title[Bifurcation set of families of complex curves]{Bifurcation set of multi-parameter families of complex curves}

%

\author{\sc Cezar Joi\c ta}
\address{Institute of Mathematics of the Romanian Academy, P.O. Box 1-764,
 014700 Bucharest, Romania and Laboratoire Europ\' een Associ\'e  CNRS Franco-Roumain Math-Mode}

\email{Cezar.Joita@imar.ro}

\author{Mihai Tib\u ar}
\address{Univ. Lille, CNRS, UMR 8524 -- Laboratoire Paul Painlev\'e, F-59000 Lille,
France}  
\email{tibar@math.univ-lille1.fr}



\thanks{The authors acknowledge the support of the Labex CEMPI
(ANR-11-LABX-0007-01). The first author acknowledges the
CNCS grant PN-III-P4-ID-PCE-2016-0330.}

\begin{abstract}
 The problem of detecting the bifurcation set of polynomial mappings $\bC^m \to \bC^k$,  $m\ge 2$, $m\ge k\ge 1$, 
 has been solved in the case $m=2$, $k=1$ only. Its solution, which goes back to the 1970s,  involves the non-constancy of the Euler characteristic of fibres. We provide here a complete answer to the general case $m= k+1 \ge 3$ in terms of  the Betti numbers of fibres and of a vanishing phenomenon discovered in the late 1990s in the real setting.
\end{abstract}

\maketitle


\section{Introduction}\label{s:intro}

The bifurcation locus of a polynomial map $f : \bC^m  \to \bC^k$ with $m\ge k$, i.e. the minimal set of points $B(f) \subset \bC^k$ outside which the mapping is a C$^\ity$ locally trivial fibration, is 
an algebraic set of dimension at most $k-1$ (cf \cite{KOS}, \cite{DRT} etc).
The following result was found in the 1970's by Suzuki \cite{Su},  and later by H\`a and L\^{e} in  \cite{HL}: 
\emph{Let $f : \bC^2 \to \bC$ be a polynomial function and let $\lambda \in \bC\m f(\Sing f)$. Then $\lambda$ is not a bifurcation value if and only if the Euler characteristic of the fibres $\chi(f^{-1}(t))$ is constant for $t$ varying in some neighborhood of $\lambda$.}

For polynomial functions of more than 2 variables,  the problem turns out to be more subtle. The Euler test is not anymore sufficient if we replace $\bC^2$ by an affine surface, as shown by Zaidenberg \cite{Za1} and Gurjar and Miyanishi \cite{GM},  or in case of polynomial maps with $m>k \ge 2$.  This led to the natural question, of which we have learned several years ago from Gurjar, whether  the constancy of the Betti numbers of the fibers was sufficient to insure topological triviality in case of regular fibres of polynomial maps $\bC^{n+1} \to \bC^{n}$, $n\ge 2$.

The aim of this paper is to characterize
the regular bifurcation values of polynomial maps $\bC^{n+1}\to\bC^n$,  $n\ge 2$, in somewhat larger generality.
We first produce an example (\S \ref{e:main} below) demonstrating that the constancy of the Betti numbers of the fibers in the neighborhood of a regular value $\lambda$ is not a sufficient condition for $\lambda$ not being a bifurcation value.   We then prove the following natural necessary and sufficient criterion:

\begin{theorem}\label{t:main1}
Let Let $f : X\to Y$ be a  morphism between two nonsingular connected affine varieties, $\dim X = n+1$, $\dim Y= n$, $n\ge 2$.
   Let $\lambda \in \im f \setminus \overline{f(\Sing f)}\subset Y$. Then  $\lambda\not\in B(f)$ if and only if
the Euler characteristic of the fibres  $f^{-1}(t)$ is constant for $t$ varying in some neighborhood of $\lambda$
and no connected component of $f^{-1}(t)$ is vanishing at infinity as $t\to\lambda$.

\end{theorem}

 In order to prove Theorem \ref{t:main1} we show the following more general result:
 
  \begin{theorem}\label{t:main1g}
 Let $p:M\to B$ be a  holomorphic map between connected complex manifolds, where $M$ is Stein and $\dim M=\dim B+1$.
We asume that the Betti numbers  $b_0(t)$ and $b_1(t)$ of all the fibers  $p^{-1}(t)$  are finite.
 Let $\lambda \in \im p \setminus \overline{p(\Sing p)}\subset B$. Then  $\lambda\not\in B(p)$ 
if and only if
the Euler characteristic of the fibres is constant for $t$ varying in some neighborhood of $\lambda$
and no connected component of $p^{-1}(t)$ is vanishing at infinity when $t\to\lambda$.
\end{theorem}

 This is in turn based on the following result for  connected fibres, the proof of which relies on deep results by Ilyashenko  \cite{Il, Il1} and Meigniez \cite{Me}.


\begin{theorem}\label{t:main2}
Let $p:M\to B$ be a  holomorphic map between connected complex manifolds which is a surjective submersion, where $M$ is Stein with $\dim M=\dim B+1$.  Assume that all the fibres of $p$ are connected and have the same finite first Betti number $b_1<\infty$.

Then $p$ is a  locally trivial fibration. 
\end{theorem}

The condition of non-vanishing components at infinity  is then employed in order to reduce the general case to Theorem \ref{t:main2}.

 We also note the following by-product of Theorem \ref{t:main2}  in the particular case of regular functions on affine surfaces, which is a small part of a result proved by Zaidenberg
\cite[Lemma 3.2]{Za1}: 

\emph{Let $f: X \to C$ be a primitive\footnote{One says that a map $f$ is \emph{primitive} if its general fibre is connected, see e.g. \cite{Za1, PT}.}  function from an affine surface $X$ to a curve $C$.
   A regular value $\lambda\in C$ is a bifurcation value of $f$ if and only if the fibre $f^{-1}(\lambda)$ is not diffeomorphic to the general fibre of $f$.}
   \smallskip
   
 An example of a non-primitive function with regular values in some neighborhood of $0$ and where the fibres are isomorphic, but which is not a fibration, is the following:
 $X\subset \bC^3$ is the surface given by the equation $[(x-1)(xz+y^2)+1][x(xz+y^2)-1] =0$ (see the main example \S \ref{e:main}),  and consider $f: X \to \bC$ as the restriction to $X$ of the linear form $l=x$ on $\bC^3$. 
  
  \section{Vanishing at infinity and the main example}\label{s:example}

 To define ``vanishing components'', we consider the more general situation of a holomorphic map $p:M\to B$ between  connected complex manifolds with $\dim M=\dim B+1$, and let $\lambda \in \im p$ be a regular value of $p$.
The fiber $p^{-1}(t)$ is a complex manifold of dimension 1 and may be not  connected;  we assume that it has finitely many  connected components.  We then denote by $C_t$ some  connected component of the fiber $p^{-1}(t)$.

\begin{definition}\label{d:vanishing}
 We say that there are \emph{vanishing components at infinity when $t$ tends to $\lambda$} if there 
 is a sequence of points $t_k \in B$,  $t_k\to \lambda$  such that for some choice of a connected component $C_{t_k}$ of $p^{-1}(t_k)$ the sequence of sets $\{C_{t_k}\}_{k\in \bN}$ is \emph{locally finite}, i.e., for any compact $K\subset M$, there is an integer $p_K\in \bN$ such that $\forall q \ge p_K$, $C_{t_q}\cap K = \emptyset$.
 
 If there is no vanishing at infinity at $\lambda$ then we say that \emph{no  connected component of $p^{-1}(t)$ is vanishing at infinity when $t\to\lambda$}, or simply that \emph{one has the property  (NV) at $\lambda$}.
\end{definition}

A more effective way to define the property  (NV) is as follows:
 we denote by $M_b$ the fibre $p^{-1}(b)$. 
Let  $M_b = \sqcup_j M_b^j$ be the decomposition of the fibre $M_b$ into connected components.
Let also $\varphi:M\to\bR$ be a continuous exhaustion function, i.e. $\{x\in M:\varphi(x)\leq r\}$
is compact for every $r\in\bR$ (for example, if $M=\bC^n$, one we may take $\varphi(x)=\|x\|$). We 
define:
\[ \mu(b)  := \max_j \inf_{x\in M_b^j} \varphi(x) \]
Then ``vanishing component at infinity when $t\to \lambda$'' means that there exists 
  a sequence of points $t_k \in B$,  $t_k\to \lambda$, such that $\lim_{k\to \infty} \mu(t_k) = \infty$.

The phenomenon of ``vanishing of components'' has been studied in the real setting of polynomials $\bR^2 \to \bR$ in \cite{TZ} and more generally in \cite{JT}. It turns out that this is related to the ``vanishing cycles'' and ``emerging cycles'' introduced for independent reasons by Meigniez \cite{Me0}, \cite{Me}. In another stream, ``vanishing cycles at infinity'' have been studied in the context of complex polynomial functions in \cite{Pa}, \cite{ST}, \cite{ST-mon}, \cite{Ti-hyp}, \cite{Ti-book} etc.

 
 A polynomial map $\bC^3 \to \bC^2$ with vanishing components and constancy of the Euler characteristic of regular fibres was produced in \cite[Remark 2.3]{HN}.
 
 Yet another phenomenon which may occur is the ``splitting at infinity at $\lambda$''  (in the terminology of \cite{TZ, JT}) when approaching the bifurcation value $\lambda$. The simplest example
is the complex polynomial $f(x,y) = x+x^2y$. 
  It was shown in the real setting \cite[Example 3.1]{TZ} that these two phenomena (vanishing and splitting at infinity) may happen simultaneously while the Betti numbers of the fibres are locally constant. 
  
 We present here the first example in the complex setting where both phenomena occur. 
 It proves at the same time that the non-vanishing hypothesis of Theorem \ref{t:main1} is necessary.

\subsection{Main example}\label{e:main}

Let $f:\bC^3\to\bC^2$ be defined by $$f(x,y,z)=\left(x,[(x-1)(xz+y^2)+1][x(xz+y^2)-1]\right).$$

The singular locus of $f$ is the union of the $z$-axis, a curve and a surface, the last two having the same image by $f$. 
One then checks that $(0,0)$ is an interior point of $\bC^2\setminus \Sing f$.

If $(a,b)\in\bC^2$ is close enough to $(0,0)$ and $a\neq 0$ then $f^{-1}(a,b)=
\{(x,y,z)\in\bC^3 \mid x=a,\ az+y^2=c_1\}\cup\{(x,y,z)\in\bC^3 \mid  x=a,\ az+y^2=c_2\}$
where $c_1$ and $c_2$ are
the two distinct roots of the equation 
$[(a-1)\lambda+1][a\lambda-1]=a(a-1)\lambda^2+\lambda -1=b$, namely:
 \[ c_{1,2}=\frac{1}{2a(a-1)}\left( -1\pm \sqrt{1+4a(a-1)(b+1)}\right) .
 \]
Therefore $f^{-1}(a,b)\simeq \bC \sqcup \bC$.


If  $b$ is close enough to zero then $f^{-1}(0,b)=\{(x,y,z)\in\bC^3 \mid x=0,\ y=d_1\}\cup\{(x,y,z)\in\bC^3 \mid x=0,\ y=d_2\}$ where $d_1$ and $d_2$ are the two roots of $y^2=b+1$; thus
$f^{-1}(0,b)\simeq \bC \sqcup \bC$.

All fibers of $f$ in a small neighborhood of $(0,0)\in\bC^2$   have therefore exactly
two connected components isomorphic to $\bC$. 

On the other hand, when
$a\to 0$ one has $c_1\to b+1$, $c_2\to \infty$, and the component corresponding to $c_1$
splits into two copies of $\bC$ whereas the component corresponding to $c_2$ vanishes at infinity.

This shows that $f$ is not locally trivial at $(0,0)$ even if its fibres are abstractly diffeomorphic  (and even isomorphic).




\section{Preliminary results}\label{s:prelim}


We first recall basic definitions, see e.g. H\" ormander \cite{Ho}, and a couple of results that we need later: 

\begin{definition}\label{d:stein}
\noindent
(a).  A complex manifold $X$ is called \emph{Stein} if there exists a smooth strictly plurisubharmonic exhaustion function $\varphi:X\to \bR$.

\noindent
(b).  Suppose that $X$ is a Stein complex manifold and $D$ is a Stein open subset of $X$.  Then $D$ is called
\emph{Runge in $X$} if the restriction morphism ${\cO}(X)\to {\cO}(D)$ has dense image in the topology of uniform convergence on compacts.
\end{definition}

\begin{proposition}\label{p:runge}\cite{Na2}
If $X$ is a Stein manifold and $\varphi:X\to \bR$ is a continuous plurisubharmonic function, then for
every $R>0$, the open subset of $X$ defined by $\{x\in X \mid \varphi(x)<R\}$ is Stein and Runge in 
$X$. \fin
\end{proposition}

\begin{proposition}
\footnote{has been extended to singular complex spaces 
of dimension 1 in \cite{Mi}.}\label{p:rungeinclusion}\cite{BS}
If $\cS$ is an open Riemann surface and $D$ is an open subset, then $D$ is Runge in $\cS$
if and only if the morphism $H_1(D,G)\to H_1({\cS},G)$ induced by the inclusion $D\hookrightarrow {\cS}$ is injective for all abelian groups $G$.  \fin
\end{proposition}

We prove the following key semi-continuity result:

\begin{proposition}\label{p:bettilc}
Let $p:M \to B$ be a holomorphic submersion, where $M$ and $B$ are connected complex manifolds and $\dim M=\dim B+1$. 
Assume that:
\begin{enumerate}
\rm \item \it  $M$ is Stein, and
\rm \item \it    the Betti numbers   $b_0(t)$ and $b_1(t)$ of the fiber $p^{-1}(t)$ are finite for all $t\in B$.
\end{enumerate}
Then the top Betti number $b_1(t)$ is lower semi-continuous.
\end{proposition}

 This looks similar to  semi-continuity results proved in case of families of hypersurfaces in  \cite{ST-defo}, \cite[Prop. 2.1]{ST-betti}, however the method of proof is totally different. 


\begin{proof}[Proof of Proposition \ref{p:bettilc}]


Let ${\mathcal C}_\lambda=\{C_1,C_2,\dots, C_{b_0}\}$ denote the set of  connected components of the fibre $p^{-1}(\lambda)$, where $b_0:=b_0(\lambda)$, for  some fixed value $\lambda\in B$.

Let $f:M\to\bC$ be a holomorphic function such that
$f_{|C_j}\equiv 2j$, for any $j\in\{1,\dots,b_0\}$.  Such functions do exist because $M$ is Stein, see e.g.  \cite[Theorem 7.4.8]{Ho}.   
We define the following disjoint open subsets of $M$:\\
$U_1=\{x\in M \mid \Re(f(x))<3\}$, $U_j=\{x\in M \mid  2j-1<\Re(f(x))<2j+1\}$ for $j\in\{2,\dots, b_{0}-1\}$, $U_{b_0}=\{x\in M \mid \Re(f(x))>2b_0+1\}$, where $\Re(f)$ denotes the real part of $f$.
It follows that  $U_j$ is Stein,  that  $U_j$ is Runge in $M$,  and that $C_j\subset U_j$, for any $j\le b_{0}$. 

One can find  a compact connected subset  $W_j$ of 
$C_j$  such that inclusion $W_j\hookrightarrow C_j$ induces a surjective morphism $H_1(W_j,\bQ)\to H_1(C_j,\bQ)$
(e.g. by choosing a finite set of compact geometric generators of  $H_1(C_j,\bQ)$
 and  some connected compact $W_j$ containing all of them). 

Next, there exists  a smooth strictly plurisubharmonic exhaustion function  $\varphi_j:U_j\to \bR$
and  a sufficiently large real number $r_j\gg 1$ 
such that  $\bB_{j} := \{y\in U_j \mid \varphi_j(y)<r_j\}$  contains $\overline W_j$, its  boundary 
$\partial \bB_{j}$   is non-singular 
at every point of  $\bB_{j}\cap C_{j}$ and it intersects  $C_j$ transversely.

We may choose $\varphi_j$ such that $\bB_{j}\cap C_{j}$ is moreover  connected. 
Indeed, let  $\Omega_j$ be the connected component of $\bB_{j}\cap C_{j}$ which contains $W_j$, and
let $s_j=\min_{x\in C_j}\varphi_j(x)$, which exists since $\varphi_j$ is an exhaustion.
We know  that $\Omega_j$ has smooth boundary. It also follows that $\Omega_j$ is Runge in $C_j$ and, since $C_j$ is a smooth complex curve, that $\overline \Omega_j$
is holomorphically convex in $C_j$ and hence in $U_j$. This last statement follows from the definition and
Theorem 7.4.8  in \cite{Ho}.
We may choose  (see e.g.  \cite[Theorem 2.6.11]{Ho})  a smooth 
plurisubharmonic function $\psi_j:U_j\to [0,\infty)$ such that $\psi_j=0$ over a neighborhood of $\overline\Omega_j$ and $\psi_{j}\geq r_j-s_j+1$ over the closure of $(\bB_{j}\cap C_{j})\setminus \Omega_j$.
We have $\{y\in C_j \mid \varphi_j(y)+\psi_j(y)<r_j\}=\Omega_j$, hence replacing
$\varphi_j$ by $\varphi_j+\psi_j$ we get 
 the claimed additional property.

By Proposition \ref{p:runge}, $\bB_{j}$ is Runge in $U_j$ and therefore 
$\bB_{j}\cap C_{j}$ is Runge in $C_j$. It  then follows from Proposition \ref{p:rungeinclusion}
that the inclusion $\bB_{j}\cap C_{j} \hookrightarrow C_{j}$  induces an injective morphism
$H_1(\bB_{j}\cap C_{j},\bQ) \to  H_1(C_j,\bQ)$. As $W_j\subset \bB_{j}$
and $H_1(W_j,\bQ)\to H_1(C_j,\bQ)$ is surjective, we deduce that $H_1(\bB_{j}\cap C_{j},\bQ) \to  H_1(C_j,\bQ)$
is in fact an isomorphism.

By Ehresmann's fibration theorem, there exists some sufficiently small  connected open ball
$B'\subset\bigcap_{j=1}^{b_0} p(U_j)\subset B$ of $\lambda$ such that the restriction:
\begin{equation}\label{eq:triv}
p_{|}: p^{-1}(B')\cap \bB_{j} \to B'
\end{equation}
is a trivial C$^\infty$-fibration.

Let us  fix  some $t\in B'$ and let ${\mathcal C}_t = \{\Gamma_1,\dots,\Gamma_q\}$ be the set of  connected components of $p^{-1}(t)$.
We define $\beta:{\mathcal C}_\lambda\to{\mathcal C}_t$ by
$\beta(C_j)=$ the unique connected component of ${\mathcal C}_t$  which intersects
$\bB_{j}$.
By changing the indices we may assume that
$\beta(C_1)=\cdots=\beta(C_{n_1})=\Gamma_1$, $\ldots,$
 $\beta(C_{n_{q'-1}+1})=\cdots=\beta(C_{b_{0}})=\Gamma_{q'}$ where $q'\leq q$ and 
$1\leq n_1<n_2<\cdots<n_{q'-1}<b_{0}$. 

From the triviality of \eqref{eq:triv} it follows that, for $1\leq j\leq n_1$, we have:
\begin{equation}\label{eq:equal}
\dim H_1\left( \bB_{j}\cap \Gamma_1 ,\bQ\right)= \dim H_1\left( \bB_{j}\cap C_j ,\bQ\right)=\dim H_1(C_j,\bQ). 
\end{equation}
Since
$\bB_{j}\cap  \Gamma_1$ is Runge in $U_j\cap \Gamma_1$ and since the latter is Runge in $\Gamma_1$, it follows that $\bB_{j}\cap  \Gamma_1$ is Runge in $\Gamma_1$.  

We need now  Narasimhan's result  \cite[Lemma 1]{Na} which says that if $X$ is a Stein space and $X_1,X_2$ are two disjoint Runge open subsets of $X$ then $X_1\cup X_2$ is Runge in $X$ if and
only if any two compact subsets $K_1\subset X_1$, $K_2\subset X_2$ can be separated by the real part of a holomorphic function $f\in{\mathcal O}(X)$.

By applying \cite[Lemma 1]{Na} we get that the union $\bigcup_{j=1}^{n_1} \bB_{j}\cap\Gamma_1$ 
is Runge in $\Gamma_1$. By Proposition \ref{p:runge},
the inclusion induces 
a monomorphism $H_1\left(\bigcup_{j=1}^{n_1} \bB_{j} \cap \Gamma_1, \bQ \right) \to   H_1(\Gamma_1, \bQ)$, which implies that  $\dim H_1(\Gamma_1,\bQ) \geq 
\sum_{j=1}^{n_1}\dim H_1(C_j,\bQ)$. Repeating this argument for $\Gamma_2,\ldots,\Gamma_{q'}$
and summing up the inequalities we obtain that 
$\sum _{j=1}^{q'} \dim H_1(\Gamma_j,\bQ) \geq b_1(\lambda)$, which implies the inequality $b_1(t)\geq b_1(\lambda)$.
\end{proof}

\section{Proof of  Theorem \ref{t:main2}}\label{s:proofth1.3}

The proof of   Theorem \ref{t:main2} falls into two steps. We first define a covering $\widetilde M$ of $M$ such that
its restriction to the fibres of $p$ is the universal covering. In the second step we use this covering and the results of  \cite{Me} to deduce that $p$ is a
locally trivial fibration.
\bigskip

\noindent
{\bf Step 1.}
Our starting point is the following key theorem by Y. S. Ilyashenko:


\begin{theorem}\label{t:ilyashenko} \cite{Il},  \cite{Il1}
Let  $M$ be a Stein manifold foliated by complex curves,  let $\widehat B$ be a transversal cross-section and let
 ${\cF}_x$ denote the leaf of the foliation which contains the point  $x\in B$.

 Then there exists
 a complex manifold $\widetilde M$  together with a locally biholomorphic
map $\pi:\widetilde M\to M$ such that the restriction $\pi_{|}:\pi^{-1}({\cF}_x)\to {\cF}_x$ is the universal covering of ${\cF}_x$ with base point $x$, for any $x\in B$.  
\end{theorem}

\begin{definition}\label{d:notations}
 In the  setting of Theorem \ref{t:main2}, let us consider the foliation defined by the fibres of the  submersion $p:M\to B$ and the cross-section $\widehat{B}$ defined by some lift by $p$ of the base $B$. 
Let us denote by $\widehat{x}$ the point of $\widehat B$ which corresponds to $x\in B$ by the identification of $\widehat B$ with $B$.
 
\end{definition}
 By applying Theorem \ref{t:ilyashenko} to this setting we obtain the manifold $\widetilde M$ and the locally biholomorphic map $\pi:\widetilde M\to M$.
We then have:

\begin{proposition}\label{l:covering}
 The map $\pi:\widetilde M\to M$ is a covering.
\end{proposition}
\begin{proof} 
 
Let us fix  some point $x\in M$ and denote $M_b :=p^{-1}(b)$,  for $b\in B$. 
As in the proof of Proposition \ref{p:bettilc}, 
 we choose a compact  connected subset $W$ of 
$M_{p(x)}$ containing  $\widehat{p(x)}$ such that inclusion $W\hookrightarrow M_{p(x)}$ induces a surjective 
morphism $H_1(W,\bQ))\to H_1(M_{p(x)}\bQ)$.  

Since $M$ is Stein, there is a strictly plurisubharmonic exhaustion function $\varphi:M\to\bR$ and a sufficiently large 
real number $R\gg 1$ 
such that the boundary $\partial \{y\in M \mid \varphi(y) \le R\}$ intersects 
the fibre $M_{p(x)}$ transversely, that $W\subset \{y\in M_{p(x)} \mid \varphi(y)<R\}$. Moreover, as in the proof of Proposition \ref{p:bettilc}, we may assume that $\{y\in M_{p(x)} \mid \varphi(y)<R\}$ is connected. 
For $b\in B$ we denote:
\[M_{b,R}:=M_b\cap\{y\in M \mid \varphi(y)<R\}.\]

 By Propositions \ref{p:runge} and \ref{p:rungeinclusion} and the inclusion $W\subset M_{p(x),R}$ we deduce that 
the inclusion $M_{p(x),R}\hookrightarrow M_{p(x)}$ induces an isomorphism $H_1(M_{p(x),R},\bQ)\simeq H_1(M_{p(x)},\bQ)$

By Ehresmann's fibration theorem, there exists a sufficiently small connected open neighborhood 
$B'\subset B$ of $p(x)$ such that the restriction $p_| :p^{-1}(B')\cap\{y\in M \mid \varphi(y)<R\} \to B'$ is a trivial C$^\infty$-fibration. 
Consequently,  for any $b\in B'$, we have the isomorphism:
\begin{equation}\label{eq:iso}
 H_1(M_{b,R},\bQ) \simeq H_1(M_{p(x),R},\bQ).
 \end{equation}

Again by Propositions \ref{p:runge} and \ref{p:rungeinclusion} we deduce that for $b\in B'$
 the inclusion $M_{b,R}\hookrightarrow M_b$ induces an injective map
$H_1(M_{b,R},\bQ)\to H_1(M_{b},\bQ)$.

According to our hypothesis, $\dim  H_1(M_{p(x)},\bQ)=\dim  H_1(M_{b},\bQ)$ and hence by \eqref{eq:iso} we get
for all $b\in B'$:
\[\dim H_1(M_{b,R},\bQ) = \dim H_1(M_{p(x),R},\bQ) =\dim H_1(M_{p(x)},\bQ) = \dim H_1(M_{b},\bQ) \]
and that the isomorphism $H_1(M_{b,R},\bQ) \simeq H_1(M_{b},\bQ)$ is induced by the inclusion.

\begin{lemma}\label{l:iso_fundam_groups}
 The inclusion $M_{b,R}\hookrightarrow M_b$ induces an 
 isomorphism $\pi_1(M_{b,R})\to \pi_1(M_{b})$, for any $b\in B'$.
\end{lemma}
\begin{proof}
As we have assumed that $M_{p(x),R}$ is connected and 
$p_| :p^{-1}(B')\cap\{y\in M \mid \varphi(y)<R\} \to B'$ is a trivial fibration, it follows that
 the open set $M_{b,R}$ is connected, for any $b\in B'$. By a Mayer-Vietoris argument using the finite dimensional homology groups of $M_b$ and  $M_{b,R}$ we show that the set $M_b \m M_{b,R}$ is a disjoint union of  finitely many cylinders, more precisely  collars over the boundary components of $\partial \overline{M_{b,R}}$, which are just circles. By retracting the collars to their boundaries we get that $M_{b,R}$ is a deformation retract of $M_b$.
\end{proof}

We return now to the proof of Proposition \ref{l:covering}.
Let us recall Ilyashenko's construction \cite{Il} of $\widetilde M$ for our foliation defined by the fibres of $p$. 

For each
$x\in M$, a point in $\pi^{-1}(x)\subset \widetilde M$ is by definition an equivalence class of paths 
$\gamma:[0,1]\to M_{p(x)}$ with fixed ends: $\gamma(0)=\widehat{p(x)}$, $\gamma(1)=x$.
The topology of $\widetilde M$ is defined as follows. For some point $\tilde x\in \widetilde M$, let
 $x:=\pi(\tilde x)$ and  choose a path $\gamma:[0,1]\to M_{p(x)}$ from $\widehat{p(x)}$ to $x$ such that $\tilde x=[\gamma]$. By considering a tubular neighborhood of $\im \gamma$ on which $p$ is a submersion and Ehresmann's theorem can be applied,
 there exists an open neighborhood $V$ of $x$ such that there exists a continuous
function $g:V\times [0,1]\to M$ with the following properties:

$\bullet$ $g(x,t)=\gamma(t)$ for any $t\in[0,1]$,

$\bullet$ $g(y,t)\in M_{p(y)}$ for any $y\in V$ and $t\in[0,1]$,

$\bullet$ $g(y,1)=y$ and $g(y,0)=\widehat{p(y)}$ for any $y\in V$.

Denoting  $\gamma_y(t):=g(y,t)$, one has by definition that:
\begin{equation}\label{eq:tilde}
 \widetilde V_g:=\{[\gamma_y] \mid y\in V\}
 \end{equation}
is  an
open neighborhood of $\tilde x$ and that the restriction $\pi_{|}: \widetilde V_g\to V$ is a homeomorphism, where the homotopy class $[\gamma_y]$ is  considered in $M_{p(x)}$ and $rel(0,1)$.
\medskip

Let now $U_0\subset M_{p(x),R}$ be a sufficiently small contractible neighborhood of $x$ and consider  a contraction $H:U_0\times[0,1]\to U_0$, i.e. $H(y,0)=y$, $H(y,1)=x$, $H(x,t)=x$ for $y\in U_0$ and $t\in[0,1]$. 
Let 
\begin{equation}\label{eq:diff}
 G=(G_1,G_2):p^{-1}(B')\cap\{y\in M \mid \varphi(y)<R\} \to M_{p(x)}\times B'
 \end{equation}
 be a diffeomorphism such that
 $G_2=p$,
$G(y)=(y,p(x))$ for each $y\in M_{p(x), R}$, and
 $G(\hat y)=(\widehat{p(x)},y)$ for each $y\in B'$.

Let then  $U:=G^{-1}(U_0\times B'$) be our neighborhood of $x$ in $M$, where $G$ is the diffeomorphism  defined at \eqref{eq:diff}.
We show in the following that $U$ is evenly covered by $\pi$.

So let $\tilde x\in \pi^{-1}(x)$. Since $\pi_1(M_{p(x),R})\to \pi_1(M_{p(x)})$ is surjective by Lemma \ref{l:iso_fundam_groups}, we can find
a path $\mu\in\tilde x$ such that $\im \mu\subset M_{p(x),R}$. Then the path $\gamma:[0,1]\to M_{p(x),R}$
defined by $\gamma(t)=\mu(2t)$ for $t\in[0,\frac 12]$ and 
$\gamma(t)=x$ for $t\in[\frac 12, 1]$  is homotopic to $\mu$,  hence $\gamma\in\tilde x$.
We define a continuous map $Z:U\times [0,1]\to M$ as follows:
$$Z(y,t)=\left
\{\begin{array}{lll}
G^{-1}(\mu(2t),p(y))&\mbox{if}&t\in[0,\frac 12],\cr
G^{-1}(H(G_1(y),2-2t),p(y))&\mbox{if}&t\in[\frac 12,1]
\end{array}
\right. $$

We have then that
$Z(y,t)\in M_{p(y)}$, $Z(y,1)=y$ and $Z(y,0)=\widehat{p(y)}$ for any $y\in U$, and 
$Z(x,t)=\gamma(t)$ for any $t\in[0,1]$. Then $\widetilde U_Z$ is a neighborhood of $\tilde x$ and the restriction
$\pi_{|}:\widetilde U_Z\to U$  is a homeomorphism, where  the notation $\widetilde U_Z$  is defined as at \eqref{eq:tilde}.
We have thus started with some point $\tilde x \in \pi^{-1}(x)$ and we have attached to it the map $Z$ and the open neighborhood $\widetilde U_Z$; since we want to compare in the following two such neighborhoods, let us denote $\widetilde U_Z$ by
$\widetilde U_{\tilde x}$ .

\smallskip
\noindent
In order to prove that $U$ is evenly covered by $\pi$, it suffices to prove the following two statements:

\begin{enumerate}
 \item  $\widetilde U_{\tilde x_1}\cap \widetilde U_{\tilde x_2}=\emptyset$ for $\tilde x_1, \tilde x_2\in \pi^{-1}(x)$, $\tilde x_1\neq \tilde x_2$,
  \item   $\pi^{-1}(U)=\bigcup _{\tilde x\in \pi^{-1}(x)}\widetilde U_{\tilde x}$.
\end{enumerate}

To prove (a) we let $\tilde y_0\in \widetilde U_{\tilde x_1}\cap \widetilde U_{\tilde x_2}$ 
where $y_0:=\pi(\tilde y_0)$.
Then there exist $\mu_1,\mu_2 :[0,1]\to  M_{p(x),R}$ with $\mu_1(0)=\mu_2(0)=p(x)$
and $\mu_1(1)=\mu_2(1)=x$ such that $[\mu_1]=\tilde x_1, [\mu_2]=\tilde x_2$, and that the paths $\nu_1,\nu_2:[0,1]\to M_{p(y_0),R}$
defined by

$$\nu_j(t)=\left
\{\begin{array}{lll}
G^{-1}(\mu_j(2t),p(y_0))&\mbox{if}&t\in[0,\frac 12],\cr
G^{-1}(H(G_1(y_0),2-2t),p(y_0))&\mbox{if}&t\in[\frac 12,1]
\end{array}
\right. $$
 (for $j= 1, 2$) represent both $\tilde y_{0}$ and hence they are homotopically equivalent in $M_{p(y_0)}$.
Moreover, they are 
homotopically equivalent in $M_{p(y_0),R}$ since the canonical morphism $\pi_1(M_{p(y_0),R})\to \pi_1(M_{p(y_0)})$ induced by the inclusion is bijective (cf Lemma \ref{l:iso_fundam_groups}).

 It follows that the two paths (for $j= 1, 2$):

$$t\rightarrow\left
\{\begin{array}{lll}
\mu_j(2t)&\mbox{if}&t\in[0,\frac 12],\cr
H(G_1(y_0),2-2t)&\mbox{if}&t\in[\frac 12,1]
\end{array}
\right. $$
are homotopically equivalent and therefore the paths $\mu_1$ and $\mu_2$ are homotopically equivalent.
Hence $\tilde x_1=\tilde x_2$ and (a) is proved.
\medskip

To prove (b)  we let $y_0\in U$ and  $\tilde y_0\in \pi^{-1}(y_0)$. Since 
the canonical morphism $\pi_1(M_{p(y_0),R})\to \pi_1(M_{p(y_0)})$ is surjective by Lemma \ref{l:iso_fundam_groups}, it follows that
there exists $\nu\in\tilde y_0$ such that $\im \nu \subset M_{p(y_0),R}$. Let 
$\mu:[0,1]\to M_{p(x),R}$ be defined by $\mu(t)=G_1(\nu(2t))$ if $t\in[0,\frac 12]$
and $\mu(t)=H(G_1(y_0),2t-1)$ if $t\in[\frac 12,1]$. Let $\tilde x=[\mu]$.  From the very definition if the following path: 
\[\nu_1(t)=\left
\{\begin{array}{lll}
\nu(4t)&\mbox{if}&t\in[0,\frac 14],\cr
G^{-1}((H(G_1(y_0),4t-1),p(y_0))&\mbox{if}&t\in[\frac 14,\frac 12],\cr
G^{-1}(H(G_1(y_0),2-2t),p(y_0))&\mbox{if}&t\in[\frac 12,1]
\end{array}
\right. \]
we get that $[\nu_1]\in U_{\tilde x}$ and moreover that 
 $\nu_1$ and $\nu$ are homotopically equivalent. Hence $\tilde y_0\in U_{\tilde x}$, which proves our claim. 
  This ends the proof of Proposition \ref{l:covering}.
\end{proof}
\bigskip

\noindent
{\bf Step 2.}
We use the following theorem proved by Meigniez, which is a nice generalization of 
a result by Palmeira \cite{Pa}.

\begin{theorem}  \cite[Corollary 31]{Me} \\
A surjective smooth submersion with all fibres diffeomorphic to $\bR^p$ is a locally trivial fibration.
\fin
\end{theorem}

In our case the fibres of $p\circ \pi:\widetilde M\to B$  are diffeomorphic to
$\bR^2$, so this theorem implies that $p\circ \pi:\widetilde M\to B$ is a locally trivial fibration.

Let us recall the definition of a Serre fibration.

\begin{definition}
A map $\pi:E\to B$ between differentiable manifolds is called a Serre fibration if for any simplex $X$,
any homotopy $f:X\times[0,1]\to B$ and $\tilde f_0$ a lifting of $f_{|X\times\{0\}}$, there exists
a lifting $\tilde f:X\times[0,1]\to E$ of $f$ such that $\tilde f_{|X\times\{0\}}=\tilde f_0$.
\end{definition}

Since, by Proposition \ref{l:covering}, $\pi:\widetilde M\to M$ is a covering and, at the same time, we have just seen that 
$p\circ \pi:\widetilde M\to B$ is a locally trivial fibration, we deduce that they both are Serre fibrations. It then follows that  
$p:M\to B$ is a Serre fibration (which is straightforward, see also \cite[Example 38]{Me}). 
\medskip

In order to conclude, we need another result by Meigniez:

\begin{proposition} \cite[Corollary 32]{Me} \\
Suppose that $E$ and $B$ are smooth manifolds such that $\dim_\bR E=\dim_\bR B+2$ and $\pi: E\to B$
is a surjective smooth map. If $\pi$ is both a submersion and a Serre fibration then $\pi:E\to B$ is
a locally trivial fibre bundle. \fin
\end{proposition}

Applying this proposition to our  submersive Serre fibration $p:M\to B$, we deduce that it is a locally trivial fibration.
 Theorem \ref{t:main2} is now proved.  \fin
 
 \smallskip
 
 Note that this also proves Theorem \ref{t:main1g} in case that the fibres $p^{-1}(t)$ of the holomorphic map $p$ are connected for any $t$ in some small neighborhood of $\lambda$.  In the next section we shall prove the general result.

\bigskip

We end this section by two simple examples which illustrate the importance of the hypotheses of
Theorem \ref{t:main2}.

\begin{example}
 Let $M=\bC^2\setminus (\{(z,w)\in\bC^2 \mid zw=1\}\cup\{0\})$ and let $p:M\to\bC$ be the projection on the first coordinate. Then $p$ is a surjective submersion, all fibres are diffeomorphic to $\bC^*$ but obviously this is not a locally trivial fibration. $M$ is not Stein.    

If $\gamma_n : [0,1]\to \bC^2$, $\gamma_n(t)=(\frac 1n,e^{2\pi it})$ for $n\geq 2$, and $\gamma:[0,1]\to \bC^2$, $\gamma(t)=(0,e^{2\pi it})$ then $\gamma_n$  is homotopically trivial
in $p^{-1}(\frac 1n)$, $\gamma_n$ converges uniformly to $\gamma$
but $\gamma$ is not homotopically trivial
in $p^{-1}(0)$.  We have here a non-trivial vanishing cycle in the terminology of Meigniez's \cite{Me}. If we apply Ilyashenko's  construction to this setting then the topology of $\widetilde M$ will be not separated.  We also have another ``cycle vanishing at infinity'' in the terminology of \cite{Ti-book}.
\end{example}
 
\begin{example}
 Let $M=\bC^2\setminus \left(\bigcup_{n\in\bN}\{(z,w)\in\bC^2  \mid w=n\}\cup\{(z,w)\in\bC^2 \mid z=w\}\right)$ and let $p:M\to\bC$ be the projection on the first coordinate. Then $p$ is a surjective submersion, all fibers are diffeomorphic to
$\bC\setminus \bN$,  and $M$ is Stein. However $p$ is not a locally trivial fibration. In the terminology of \cite{Me},  non-trivial emergent cycles occur around the origin. This satisfies Ilyashenko's theorem but not the  finiteness of the Betti numbers of fibres condition  of  Theorem \ref{t:main2}.

\end{example}


  
\section{Proof of Theorem \ref{t:main1g}}
 

After having proved Theorem \ref{t:main1g} in the particular case of connected fibres (Theorem \ref{t:main2}), let us now treat the general case.
 
 Let $p:M\to B$ be a  holomorphic map between connected complex manifolds such that $M$ is Stein, $\dim M=\dim B+1$, and
the Betti numbers  $b_0(t)$ and $b_1(t)$ of all the fibres  $p^{-1}(t)$  are finite.
Let $\lambda \in \im p \setminus \overline{p(\Sing p)}\subset B$. We assume that
the Euler characteristic of the fibres is constant for $t$ varying in some neighborhood of $\lambda$
and no connected component of $p^{-1}(t)$ is vanishing at infinity when $t\to\lambda$.

Let $\cC_t$ be the set of  connected components of the fibre $p^{-1}(t)$, which is by hypothesis a finite set with $b_0(t)$ elements, 
and let $b_0(\lambda)=b_0$.

 Let $\{ C_1, \ldots , C_{b_0}\}$ denote the  connected components of $p^{-1}(\lambda)$.
 Choose and fix some point $x_j\in C_j$, for any $j$. 
 Since $p$ is a submersion over some small neighborhood of $\lambda$, we may choose some mutually disjoint small balls 
$B_j\subset \bC^{n+1}$  centered at $x_j$, and some sufficiently small ball $D\subset B$
 centered at $\lambda$ such that the intersection $B_j \cap p^{-1}(\lambda)$ is  connected and that
  the restriction 
 \begin{equation}\label{eq:fib}
  p_| : B_j \cap p^{-1}(D) \to D
 \end{equation}
 
is a trivial C$^\infty$-fibration for any $j\in \overline{1, b_0}$.

We then define $a_t(j)$ as the  connected component of $p^{-1}(t)$ which intersects $B_ j \cap p^{-1}(D)$. This yields a well defined  function $a_t : \{ 1, \ldots , b_0\} \to \cC_t$ since  $B_j \cap p^{-1}(t)$ is itself  connected for any $t\in D$
and any $j$, by \eqref{eq:fib}.

\begin{lemma}\label{l:surj}
The Betti numbers $b_0(t)$ and  $b_1(t)$ of the fibre $p^{-1}(t)$ are constant for $t$ in some  small neighborhood of $\lambda$.
\end{lemma}
\begin{proof}
We claim that the function $a_t$ is surjective,  for any $t$ in some small neighborhood of $\lambda$.  If not so, then there is a sequence $t_i \to \lambda$ and some  connected component $C^{t_i}$ of $p^{-1}(t_i)$ which does not intersect any ball $B_j$, for $j\in \overline{1, b_0}$.
By the assumed non-vanishing condition (NV) at $\lambda$, there exists some compact $K\subset M$
such that $C^{t_{i}}\cap K \not= \emptyset$ for any $i\gg 1$. (We actually get this possibly after passing to an infinite sub-sequence of sets $C^{t_{i_p}}$.)
Let us then fix some point $z_{t_{i}} \in C^{t_{i}}\cap K$. We may assume, again after passing to a sub-sequence, that $z_{t_{i}}$ tends to some limit point $z_0 \in K$. Consequently,  $\lim_{i\to \infty}p(z_{t_{i}}) = p(z_0) = \lambda$. Since $z_0 \in C_1 \cup \ldots \cup C_{b_0}$, there is some $j_0\in \overline{1, b_0}$ such that $z_0 \in C_{j_0}$.  By taking a simple path linking $x_{j_0}$ to $z_0$ within $C_{j_0}$ and by considering a tubular neighborhood of it in $M$, we may assume, by using Ehresmann's theorem, that there exists a  connected open set $\cN_{j_0}$ containing $z_0$ and $x_{j_0}$ such that  the intersection $\cN_{j_0} \cap p^{-1}(\lambda)$ is  connected and that
the restriction:
\begin{equation}\label{eq:trivfib}
 p_| :\cN_{j_0} \cap p^{-1}(D) \to D 
 \end{equation}
is a trivial C$^\infty$-fibration (possibly after shrinking the ball $D$), and therefore all its fibres are  connected.

Since $\cN_{j_0}$ is open and contains $z_0$,  it contains the points $z_{t_{i}}\in C^{t_i}$ for sufficiently large index $i\in \bN$. Hence for $i\gg 1$ the connected set $p^{-1}(t_{i}) \cap \cN_{j_0}$ is contained in $C^{t_i}$ and on the other hand  
$B_{j_0}\cap \cN_{j_0} \cap p^{-1}(D) \not= \emptyset$.

Since we have assumed that $C^{t_i} \cap B_{j_0} = \emptyset$,  we thus get a contradiction. Our surjectivity claim is proved.  In particular, we have  $b_0(t)\leq b_0(\lambda)$ for $t$ in some small neighborhood of $\lambda$.

 By Proposition \ref{p:bettilc}, possibly after shrinking $D$, 
we also have $b_1(t)\geq b_1(\lambda)$.   We therefore get:
\[ \chi(t)=b_0(t)-b_1(t)\leq b_0(\lambda)-b_1(t)\leq b_0(\lambda)-b_1(\lambda)=\chi(\lambda)\]
and since $\chi(\lambda) = \chi(t)$ by our hypothesis, we
 deduce that the Betti numbers $b_0(t)$ and $b_1(t)$ are constant at $\lambda$.  
\end{proof}

In particular, from the equality $b_0(t)=b_0(\lambda)$ proved in Lemma \ref{l:surj} we deduce that the function $a_t$ is bijective, for $t$ in some sufficiently small neighborhood of $\lambda$.
 We may therefore denote from now on $C_j(t) := a_t(j)$, $j\in \overline{1, b_0}$.
 
 Keeping the above notations, we moreover have the following result showing that the connected components belong to families:

\begin{lemma}\label{l:open}
 The union $V_j := \bigcup_{t\in D} C_j(t)$ is an open set in $M$, for any $j\in \overline{1, b_0}$.
\end{lemma}

\begin{proof}
 Let $y_0 \in C_j(t_0)$ for some $t_0\in D$. Let $\hat D$ be the section of the trivial  fibration \eqref{eq:fib} which contains the center  $x_j\in C_j(\lambda)$ of the ball $B_j$, and let $z_0 :=  C_j(t_0)\cap \hat D$.  We consider a simple path within $C_{j}(t_0)$ connecting $z_{0}$ to $y_0$.  Like in the proof of Lemma \ref{l:surj}, one may find a connected open tubular neighborhood $\cN_{j}\subset M$ of this  path such that   $\cN_{j} \cap p^{-1}(t_0) = \cN_{j} \cap C_j(t_0)\subset V_{j}$ is  connected and that the restriction:
 \begin{equation}\label{eq:fibtriv2}
 p_| :\cN_{j} \cap p^{-1}(D_0) \to D_0
 \end{equation}
  is a trivial C$^\infty$-fibration, where $D_0\subset D$ is some sufficiently small ball centered at $t_0\in \bC^n$. 
 Since $\cN_{j}$ contains by definition some neighborhood of  $z_0$,  the intersection $\cN_{j}\cap \hat D$ contains a neighborhood of  $z_0$ in $\hat D$.  It then follows from the definition of the components $C_j(t)$ that the fibres of the fibration \eqref{eq:fibtriv2} must be subsets of $V_j$,  after possibly shrinking $D_0$.  And since the neighborhood $\cN_{j} \cap p^{-1}(D_0)$ is included in $V_j$, it contains by definition a small ball around $y_0$. 
\end{proof}

We  may now finish  the proof of  Theorem \ref{t:main1g}.
By the definition of $V_j$ and the fact that the functions $a_t$ are bijections for all $t\in D$, we get that the  union $\bigcup_{j=1}^{b_0} V_j$ is a disjoint union of connected components and it is equal to $p^{-1}(D)$, and that 
$p(V_j)=D$, $\forall j\in \overline{1,b_0}$. 
Since $D$ is Stein it follows that each $V_j$ is Stein. Therefore $p_{|V_j}:V_j\to D$ is a submersion with connected 1-dimensional fibres $C_j(t)$  on the Stein manifold $V_j$.

By Proposition \ref{p:bettilc} there exists an open neighborhood $D'$ of $\lambda$, $D'\subset D$, such that we have $b_1(C_j(t))\geq b_1(C_j)$ for $t\in D'$. We may assume without loss of generality that this holds for the same $D'$ for any $j\in \overline{1,b_0}$.
Since we have proved that the total first Betti number is constant, i.e.  $\sum_{j=1}^{b_0} b_1(C_j(t))=\sum_{j=1}^{b_0} b_1(C_j)$, we obtain the constancy on every component, i.e.   $b_1(C_j(t))= b_1(C_j)$, $\forall j\in \overline{1,b_0}$.
We then apply Theorem \ref{t:main2} to the restriction $p_{|V_j \cap p^{-1}(D')}$ and get that it is a trivial fibration $\forall  j\in \overline{1,b_0}$. Therefore $p$ is a trivial fibration.


\section{Some consequences of the preceding results}\label{s:conseq}

\begin{corollary}\label{t:main3}
Let $p:M\to B$ be a holomorphic map between  connected complex manifolds, where $M$ is Stein  with $\dim M=\dim B+1$, and the components of the fibres of $p$ have finitely generated fundamental group.
Let $\lambda \in \im p \setminus \overline{p(\Sing p)}\subset B$.
\begin{enumerate}
\rm \item \it Assume that for $t$ in some neighborhood of $\lambda$ the fibres $p^{-1}(t)$ have constant  Euler characteristic  and the  Betti number $b_1$ of each of their components is at least 2. Then $p$ is a locally trivial fibration at $\lambda$.

\rm \item \it
Assume that for $t$ in some neighborhood of $\lambda$ the fibres $p^{-1}(t)$ have constant 
 Betti numbers $b_0(t)$ and $b_1(t)$  and the  Betti number $b_1$ of each of their components is at least 1.  Then $p$ is a locally trivial fibration at $\lambda$.
\end{enumerate}
\end{corollary}

\begin{proof}

(a).  
We  use the construction and the notations in the proof of Proposition \ref{p:bettilc}, from which we  recall the following:
let $\mathcal C_\lambda = \{C_1, \ldots, C_{b_0}\}$ be the set of  connected components of the fibre $p^{-1}(\lambda)$, where $b_0:=b_0(\lambda)$. We chose a holomorphic function $f:M\to\bC$ such that $f_{|C_j}\equiv 2j$, for any $j\in\{1,\dots,b_0\}$.
We defined $U_1=\{x\in M \mid \Re(f(x))<3\}$, $U_j=\{x\in M \mid  2j-1<\Re(f(x))<2j+1\}$ for $j\in\{2,\dots, b_{0}-1\}$, $U_{b_0}=\{x\in M \mid \Re(f(x))>2b_0+1\}$, where $\Re(f)$ denotes the real part of $f$.
We chose smooth strictly plurisubharmonic exhaustion functions  $\varphi_j:U_j\to \bR$, 
sufficiently large real numbers $r_j\gg 1$, and a connected open neighborhood 
$B'$ of $\lambda$ such that the restrictions $p_| :p^{-1}(B')\cap\{y\in U_j \mid \varphi_j(y)<r_j\} \to B'$ are trivial C$^\infty$-fibrations with connected fibres.


Fixing  some $t\in B'$,  let ${\mathcal C}_t = \{\Gamma_1,\dots,\Gamma_q\}$ be the set of  connected components of $p^{-1}(t)$. We have defined the function $\beta:{\mathcal C}_\lambda\to{\mathcal C}_t$ by
$\beta(C_j):=$ the unique connected component of ${\mathcal C}_t$  which intersects
$\{y\in U_j \mid \varphi_j(y)<r_j\}$.
By changing the indices we may assume that
$\beta(C_1)=\cdots=\beta(C_{n_1})=\Gamma_1$, $\ldots,$
 $\beta(C_{n_{q'-1}+1})=\cdots=\beta(C_{b_{0}})=\Gamma_{q'}$ where $q'\leq q$ and 
$1\leq n_1<n_2<\cdots<n_{q'-1}<b_{0}$.  
Finnaly, in the proof of Proposition \ref{p:bettilc} we have obtained the inequality: $\sum _{j=1}^{q'}b_1(\Gamma_j)\geq b_1(\lambda)$.

 Since by our hypothesis  $b_1(\Gamma_j)\geq 2$ 
for every $j\in\{1,\dots,q\}$ and, in particular, for $j\in\{q'+1,\dots,q\}$, we obtain that
$b_1(t)\geq b_1(\lambda)+2(q-q')$.  By our hypothesis  we have $\chi(t)=\chi(\lambda)$ and hence
$b_1(t)=b_1(\lambda)-b_0+q$, since $q=b_0(t)$  by definition. This yields $b_1(\lambda)-b_0+q\geq b_1(\lambda)+2(q-q')$
and hence $2q'\geq b_0+q$. 

Since by our construction we have $q'\leq q$ and $q'\leq b_0$, we  conclude that the equalities $b_0=q$ and $q=q'$ must hold.
This shows in particular that we have ``non-vanishing of components'' at $\lambda$. 
We may now apply Theorem \ref{t:main1g} to conclude.


\noindent
(b). The same arguments as at (a), but  instead of $b_1(\Gamma_j)\geq 2$ we have now by hypothesis  $b_1(\Gamma_j)\geq 1$ and $b_0=q$.
We thus obtain the inequality $b_1(t)\geq b_1(\lambda)+(b_0-q')$. 

Since $b_1(t)=b_1(\lambda)$ by hypothesis, we get $q'\geq b_0$.
Since $b_0\geq q'$ by construction, we must have equality $b_0=q'$, 
thus no connected component vanishes at infinity,  and we conclude by applying
 Theorem \ref{t:main1g}.

\end{proof}

For  a polynomial map $f=(f_1,\dots,f_n):\bC^{n+1}\to\bC^n$,  we denote by $\chi(t)$ the  Euler characteristic
of the fibre $F_{t}:=f^{-1}(t)$ and by
 $\overline F_t$ the closure of  $F_{t}$ in $\bP^{n+1}$, in some fixed system of coordinates of $\bC^{n+1}$.

\begin{corollary}
Let $f :\bC^{n+1}\to\bC^n$ be a polynomial map and let $\lambda \in \im f \setminus \overline{f(\Sing f)}\subset \bC^{n}$. 
 If the degree  $\deg \overline F_t$ and the Euler characteristic $\chi(t)$ are constant for $t$ varying in  some neighborhood of $\lambda$,  
then $\lambda \not\in B(f)$.
\end{corollary}

\begin{proof}  
Since $\deg \overline F_t$ is constant, it follows that the number of intersection points $F_t \cap H$, where $H\subset \bC^{n+1}$ is  a general hyperplane,  is constant, for all $t$ close to $\lambda$. We show that this implies  (NV) at $\lambda$. 

By our hypothesis, the hyperplane $H$ intersects all components of $F_t$. If not (NV) at $\lambda$, then there is a component which is ``lost''
at infinity, hence  there is at least some point $p\in F_t \cap H$ which  tends to infinity
when $t\to \lambda$, which just means that the restriction $f_{| H}$ is not proper. The number of intersection points $F_t \cap H$ is a locally semicontinuous function in the sense that this number cannot locally increase in the limit whenever the limit of points exists in $F_{\lambda}$. In our case there is at least a path $\gamma$  consisting of such intersection points $\gamma(t)\in F_t \cap H$, such that $\| \gamma(t)\| \to \ity$ as $t\to \lambda$. This yields the contradiction. We may conclude by applying Theorem \ref{t:main1}.
\end{proof}

\begin{remark} The above statement has been proved in \cite[Corollary 2.2]{HN}  in the particular case when the zero locus at infinity $Z((f_1)_{d_{1}},\ldots , (f_n)_{d_{n}})\cap H_\infty$ of the homogeneous parts $(f_i)_{d_{i}}$ of the highest degree $d_i$ of $f_{i}$ is of dimension zero.
In this case we have the constancy of the degree $\deg(\overline F_t)= d_{1}\cdots d_{n}$ for every $t$
in a neighborhood of $\lambda$.
\end{remark}

\begin{remark} 
The following example shows that one can have local triviality without the constancy of the degree of the fibres.  Let $f:\bC^3\to\bC^2$,   $f(x,y,t)=(y+tx^2,t)$.
Then $f$ is a locally trivial fibration: indeed,  $g:\bC^3\to \bC^3$,  $g(u,v,s)=(u,v-su^2,s)$,  is an isomorphism of inverse
 $g^{-1}(x,y,z)=(x,y+tx^2,t)$, and one has
$f\circ g(u,v,s)=(v,s)$. Computing the fibres of $f$, one has
 $f^{-1}(0,0)=\{y=t=0\}$ of degree 1, and
for $\alpha\neq 0$, $f^{-1}(0,\alpha)=\{(x,y,t) \mid t=\alpha,y+\alpha x^2=0\}$ of degree 2.
\end{remark}



\begin{thebibliography}{999}

\bibitem[BS]{BS} H. Behnke, K. Stein, \emph{Entwicklung analytischer Funktionen auf Riemannschen Fl\" achen. }
Math. Ann. 120 (1949). 430--461.

\bibitem[Br]{Br}
S.A. Broughton, 
\emph{On the topology of polynomial hypersurfaces.} Singularities, Part 1 (Arcata, Calif., 1981), 167--178,
Proc. Sympos. Pure Math., 40, Amer. Math. Soc., Providence, RI, 1983. 


\bibitem[DRT]{DRT}
L.R.G. Dias, M.A.S. Ruas, M. Tib\u ar, \textit{Regularity at infinity of real mappings and a Morse-Sard theorem}, J. Topology, 5 (2012), no. 2, 323--340.

\bibitem[GM]{GM}
 R.V. Gurjar, M. Miyanishi,  \emph{Automorphisms of affine surfaces with A1-fibrations.} Michigan Math. J. 53 (2005), no. 1, 33--55.

\bibitem[HL]{HL}
 H\`a H.V.,   L\^e D.T., \emph{Sur la topologie des polyn\^omes complexes}, 
Acta Math. Vietnam. 9 (1984), no. 1, 21--32.

\bibitem[HN]{HN}
 H\`a H.V.,  Nguyen T.T.,  \emph{On the topology of polynomial mappings from $\bC^n$ to $\bC^{n-1}$.} Internat. J. Math. 22 (2011), no. 3, 435--448.
 
 \bibitem[Ho]{Ho} 
  L. H\"ormander,  An introduction to complex analysis in several variables. D. Van Nostrand Co., Inc., 
Princeton, N.J.-Toronto, Ont.-London 1966.

\bibitem[Il1]{Il}  
Y. S. Ilyashenko,  \emph{Covering manifolds for analytic families of leaves of foliations by analytic curves.} Topol. Methods Nonlinear Anal. 11 (1998), no. 2, 361--373. 
 
\bibitem[Il2]{Il1} 
 Y. S. Ilyashenko,  \emph{Foliations by analytic curves.} Mat. Sb. (N.S.) 88 (130) (1972), 558--577.

\bibitem[JT]{JT} 
C. Joi\c ta, M. Tib\u ar, \emph{Bifurcation values of families of real curves},  
Proc. Royal Soc. Edinburgh Sect.A 147,  6 (2017), 1233--1242.

\bibitem[KOS]{KOS}
 K. Kurdyka, P. Orro, S. Simon, {\em Semialgebraic Sard theorem for generalized critical values,} J. Differential Geometry 56 (2000), 67--92.
 
 \bibitem[Me1]{Me0} G. Meigniez, \emph{Sur le rel\`evement des homotopies.}  C. R. Acad. Sci. Paris S\'er. I Math. 321 (1995), no. 11, 1497--1500.
 
\bibitem[Me2]{Me} G. Meigniez, \emph{Submersions, fibrations and bundles.} Trans. Amer. Math. Soc. 354 (2002), no. 9, 3771--3787.

\bibitem[Mi]{Mi}
N. Mihalache, \emph{The Runge theorem on 1-dimensional Stein spaces.} Rev. Roum. Math. 
Pures Appl. 33 (1988), no. 7, 601--611.

\bibitem[Na1]{Na} R. Narasimhan, \emph{Imbedding of Holomorphically Complete
Complex Spaces.} Amer. J. Math. 82 (1960), 917--934.

\bibitem[Na2]{Na2} R. Narasimhan, \emph{The Levi problem for complex spaces. II.} Math. Ann. 146 (1962), 195-- 216.


\bibitem[Pa]{Pa} C.F.B. Palmeira,  \emph{Open manifolds foliated by planes}, Ann. of Math. 107 (1978), 109--131.

\bibitem[PT]{PT} 
A.J. Parameswaran, M. Tib\u ar,  \emph{On the geometry of regular maps from a quasi-projective surface to a curve.} Eur. J. Math. 1 (2015), no. 2, 302--319. 

\bibitem[Par]{Par} 
A. Parusi\' nski,  \emph{On the bifurcation set of complex polynomial with isolated singularities at infinity.} Compositio Math. 97 (1995), no. 3, 369--384.


\bibitem[ST1]{ST}
 D. Siersma,  M. Tib\u ar, \emph{Singularities at infinity
and their vanishing cycles},  Duke Math. Journal 80 (3) (1995), 771--783.

\bibitem[ST2]{ST-mon}
 D. Siersma,  M. Tib\u ar, \emph{Singularities at infinity and their vanishing cycles. II. Monodromy.} Publ. Res. Inst. Math. Sci. 36 (2000), no. 6, 659--679.

\bibitem[ST3]{ST-defo}
 D. Siersma,  M. Tib\u ar, \emph{Deformations of polynomials, boundary singularities and monodromy.} Dedicated to Vladimir I. Arnold on the occasion of his 65th birthday. Mosc. Math. J. 3 (2003), no. 2, 661--679.

\bibitem[ST4]{ST-betti}
 D. Siersma,  M. Tib\u ar, \emph{Betti bounds of polynomials.} Mosc. Math. J. 11 (2011), no. 3, 599--615.


\bibitem[Su]{Su} M. Suzuki,
\emph{Propri\'et\'es topologiques des polyn\^omes de deux variables complexes, et automorphismes alg\'ebriques de l'espace $\bC^2$.} 
J. Math. Soc. Japan 26 (1974), 241--257.

\bibitem[Ti1]{Ti-hyp}
M. Tib\u ar, \emph{Vanishing cycles of pencils of hypersurfaces}. Topology 43 (2004), no. 3, 619--633.

\bibitem[Ti2]{Ti-book}
M. Tib\u ar, Polynomials and vanishing cycles, Cambridge Tracts in Mathematics, 170, 
 Cambridge University Press 2007.
 
\bibitem[TZ]{TZ} 
M. Tib\u ar,  A. Zaharia, \emph{Asymptotic behavior of families of real curves}. Manuscripta Math. 99 (1999), no. 3, 383--393.




\bibitem[Za1]{Za1}
M.G. Zaidenberg, 
\textit{Isotrivial families of curves on affine surfaces, and the characterization of the affine plane}. Izv. Akad. Nauk SSSR Ser. Mat. 51 (1987), no. 3, 534--567; translation in
Math. USSR-Izv. 30 (1988), no. 3, 503--532.  

\bibitem[Za2]{Za2}
M.G. Zaidenberg, 
\textit{Additions and corrections to the paper: ``Isotrivial families of curves on affine surfaces, and the characterization of the affine plane''}.  Izv. Akad. Nauk SSSR Ser. Mat. 55 (1991), no. 2, 444--446; translation in
Math. USSR-Izv. 38 (1992), no. 2, 435--437. 



\end{thebibliography}
\end{document}